\documentclass[a4paper,12pt]{article}
\title{On Canonical Homomorphisms of Tensor Sheaves}
\usepackage{mathrsfs}
\usepackage{xypic}
\usepackage{amsthm}
\usepackage{amsmath}
\usepackage{amscd}
\usepackage{amssymb}
\usepackage{txfonts}
\usepackage[all]{xy}
\usepackage[OT2,T1]{fontenc}

\DeclareSymbolFont{cyrletters}{OT2}{wncyr}{m}{n}
\DeclareMathSymbol{\Sha}{\mathalpha}{cyrletters}{"58}

\date{}
\author{Jianke Chen\footnote{ Department of Mathematics, Capital Normal University, Beijing 100048,China.
\quad Email:jkchen003@gmail.com\quad Tel:+86-13811897120}}
\begin{document}
\maketitle \theoremstyle{definition}
\newtheorem{theorem}{Theorem}[section]
\newtheorem{example}[theorem]{Example}
\newtheorem{definition}[theorem]{Definition}
\newtheorem{cor}[theorem]{Corollary}
\newtheorem{prop}[theorem]{Proposition}
\newtheorem{remark}[theorem]{Remark}
\newtheorem{lemma}[theorem]{Lemma}
\newtheorem{exercise}[theorem]{Exercise}
\newtheorem{fact}[theorem]{F\/ACT}

\begin{abstract}
In this paper we define tensor modules(sheaves) of {\sl Schur type},
or of {\sl generalized Schur type} associated with the give module
(sheaf), using the so-called {\sl Schur} functors. Then using global
method we construct canonical homomorphisms between these
modules(sheaves). We will get canonical isomorphisms if the original
sheaf is locally free using idea of algebraic geometry, which is in
fact a generalization of result in linear algebra. In the final
section, we give canonical complexes using homomorphisms constructed
before, and these complexes will become split exact sequence if
further condition holds. And we could use local method to give
concrete descriptions of these canonical homomorphisms.
\end{abstract}

\textbf{Keywords:}

Tensor Sheaves of {\sl Schur} Types, Generalized isomorphism on
determinant, Canonical Homomorphisms,Canonical Split Sequences.


\vskip 10pt \textbf{Acknowledgements:}

I wish to thank Prof.Ke-zheng Li, Prof. Zi-feng Yang for their
heuristic discussions with me on this topic.

\section{Introduction}

Tensor Power, Symmetric Power and Wedge Power are basic tools used
in mathematics. We are going to list its applications in linear
algebra, representation theory and algebraic geometry.

For linear algebra, let~$V$~be a~$n$~dimensional vector space
over~$k$, and~$f\in\,End_k(V)$, then using functorial property of
tensor product, we may get the following
homomorphism~$T_k^r(f):T_k^r(V)\rightarrow\,T_k^r(V)$. Similar
homomorphisms can also be induced for symmetric power and wedge
power. Then there exists the following
equality~$det(T_k^r(f))=(det(f))^{rn^{r-1}}$,
$det(S_k^r(f))=(det(f))^{\frac{(n+r-1)!}{n!(r-1)!}}$~and~$det(\wedge_k^r(f))=
(det(f))^{\frac{(n-1)!}{(r-1)!(n-r)!}}$. A special case of the
statements before is the following equality:
$$\left|\begin{array}{cccc}a^3&3a^2b&3ab^2&b^3\\a^2c&a^2d+2abc&b^2c+2abd&b^2d\\
ac^2&bc^2+2acd&ad^2+2bcd&bd^2\\c^3&3c^2d&3cd^2&d^3\end{array}\right|
=\left|\begin{array}{cc}a&b\\c&d\end{array}\right|^6$$
where~$a,b,c,d$~are elements of commutative ring~$R$, see \cite{L1}
or \cite{L2} for details. In these references, the proof is an
algebraic proof, in this place we are going to prove similar results
into a more general fact using geometric method.

Next we are going to state some applications of these operations in
representation theory. For a given representation, we may use these
operations to construct new representations, such as the basic
example that if~$V$~is a finite dimensional complex vector space,
then we have a canonical~$GL(V)-$invariant
decomposition~$V\otimes\,V\cong\,Sym^2(V)\oplus\wedge^2(V)$, and
more generally, we may
write~$V\otimes\,V\otimes\,V\cong\,Sym^3(V)\oplus\wedge^3(V)\oplus\mbox{\
another spaces}$, see \cite{FH} or \cite{S} for details. Schur
functors are generalized symmetric powers and wedge powers used in
construction of new representations. As a general concrete
example,there admits the following
decomposition,~$Sym^n(Sym^2V)\cong\bigoplus\limits_{i=0}^{[\frac{n}{2}]}Sym^{2n-4i}V$~
One basic fact for complex representation theory is
that~$\mathbb{C}$~is an algebraic closed field, so if we want to
extend these results into general base, similar results may NOT
globally hold, which means these properties(isomorphisms, or
equalities) hold if we add some localization property.

Finally these operations occur in algebraic geometry, such as
considering the Grassmanian space~$\mathbb{G}(n,m)$, which means the
set of~$m$~dimensional subspaces of a given~$n$~dimensional vector
space. The standard procedure is taking~$r$-th wedge of~$V$, we
denote this space by~$\wedge^r(V)$. In this method,
every~$m$~dimensional subspace becomes~$1$~dimensional vector space,
then as well known, the set of these~$1$~dimensional subspaces
becomes a geometric object, namely~$P(\wedge^r(V))$, or
equivalently~$Proj(S(\wedge^r(V)))$. In this construction, symmetric
power and wedge power occurs naturally.

%In construction of blow-ups, the method we usually used

In this paper, we will define tensor modules (resp. sheaves) of {\sl
schur type}, and more generally, {\sl generalized Schur type}
associated with general~$R-$modules (resp. $\mathscr{O}_X-$modules),
and investigate canonical homomorphisms between these kind of
modules. In section 2, using the idea of algebraic geometry, we will
give a geometric global canonical isomorphism of such kind of
modules, if the original module is locally free. In section 3, we
will study several kinds of canonical homomorphisms of tensor
sheaves. Note that for general tensor modules of {\sl Schur type},
we can not use local coordinates as for the locally free case. The
reason is for locally free sheaves, the gluing data information is
reflected by a~$n\times\,n$~invertible matrix, but for general
sheaves, we know nothing about the gluing data information. The main
idea for our method is that tensor sheaves of {\sl Schur type} admit
canonical permutation group's action. In fact, this method is
already used in Prof. ke-zheng Li's lectures on de Rham Complexes.
For details, see \cite{L3}.

In the final section, we will give two concrete examples of general
theorems, in which case that these canonical homomorphisms will make
complexes. And under some further assumption, these complexes will
become split exact sequences. We may use permutations to write these
canonical homomorphisms clearly, and we will use local method
(locally defined homomorphisms, then check these homomorphisms are
independent of the given basis), and check that it it the same with
global method.

\section{SOME PREPARATIONS}

Let~$R$~be a commutative ring with 1, and~$M$~be an~$R$-module.
Using tensor product one can construct~$R$-modules with~$M$~in the
following ways:

\vskip 5pt\leftskip=30pt\parindent=-10pt i) symmetric product
(~$S_R^n(M)$~for any~$n\in\mathbb{N}$~);

ii) wedge product (~$\wedge_R^n(M)$~for any~$n\in\mathbb{N}$~);

iii) tensor product over~$R$.

\vskip 5pt\leftskip=0pt\parindent=20pt One can use the above
constructions repeatedly, for a finite number of (but arbitrarily
many) times. For example,
\begin{equation}\label{f1}\mathfrak{S}(M)=S_R^r(S_R^n(M)\otimes_R\wedge_R^m(M))^{\otimes_R2}\end{equation}
Such a construction procedure is called a {\sl Shur type}. In many
cases (e.g in algebraic geometry or representation theory) such kind
of complex construction procedures of modules appear.

Note that any Shur type is canonical, i.e. for any Shur type
$\mathfrak{S}$, any $R$-module homomorphism $f:M\to N$ induces an
$R$-module homomorphism
$\mathfrak{S}(f):\mathfrak{S}(M)\to\mathfrak{S}(N)$ canonically.

For any Shur type $\mathfrak{S}$, the construction procedure gives
an epimorphism $M^{\otimes_Rd}\to\mathfrak{S}(M)$. We call $d$ the
{\sl degree} of $\mathfrak{S}$, denoted by $d(\mathfrak{S})$. For
example, it is not hard to see the Shur type in (\ref{f1}) has
$d(\mathfrak{S})=2r(n+m)$.

In some cases one also uses the following way of construction in
addition to i-iii):

\vskip 5pt\leftskip=30pt\parindent=-10pt iv) homomorphism module
$Hom_R(\cdot ,\cdot )$,

\vskip 5pt\leftskip=0pt\parindent=20pt especially for locally free
modules of finite rank. For example,
\begin{equation}
\label{f2}\mathfrak{S}'(M)=S_R^r(S_R^n(M)\otimes_RHom_R(\wedge_R^m(M),M^{\otimes_R2})
\end{equation}
Such a construction procedure is called a {\sl generalized Shur
type}.

Let $X$ be a scheme and $\mathscr{F}$ be a quasi-coherent sheaf on
$X$. Then any Shur type $\mathfrak{S}$ defines a quasi-coherent
sheaf $\mathfrak{S}(\mathscr{F})$, because $\mathfrak{S}$ is
canonical. For example, if the Shur type $\mathfrak{S}$ is as in
(\ref{f1}), then
$$\mathfrak{S}(\mathscr{F})=S_{\mathscr{O}_X}^r(S_{\mathscr{O}_X}^n(\mathscr{F})\otimes_{\mathscr{O}_X}
\wedge_{\mathscr{O}_X}^m(\mathscr{F}))^{\otimes_{\mathscr{O}_X}2})$$
Furthermore, if $f:\mathscr{F}\to\mathscr{F}'$ is a homomorphism of
quasi-coherent sheaves on $X$, then $f$ induces a homomorphism
$\mathfrak{S}(f):\mathfrak{S}(\mathscr{F})\to\mathfrak{S}(\mathscr{F}')$
canonically.

If $\mathfrak{S}'$ is a generalized Shur type and $\mathscr{F}$ is
locally free of finite rank, one can also define
$\mathfrak{S}'(\mathscr{F})$. For example, if the Shur type
$\mathfrak{S}'$ is as in (\ref{f2}), then
$$\mathfrak{S}'(\mathscr{F})=S_{\mathscr{O}_X}^r(S_{\mathscr{O}_X}^n(\mathscr{F})\otimes_{\mathscr{O}_X}
\mathscr{H}om_{\mathscr{O}_X}(\wedge_{\mathscr{O}_X}^m(\mathscr{F}),\mathscr{F}^{\otimes_{\mathscr{O}_X}2})$$

\begin{definition}
Let $R$ be a commutative ring with 1, and $M$ be an $R$-module. By a
{\sl tensor module} (resp. {\sl generalized tensor module}) of $M$
over $R$ we mean a direct sum
$\mathfrak{S}(M)=\mathfrak{S}_1(M)\oplus\cdots\oplus\mathfrak{S}_n(M)$
(resp.
$\mathfrak{S}'(M)=\mathfrak{S}_1^{\prime}(M)\oplus\cdots\oplus\mathfrak{S}_n^{\prime}(M)$),
for some Shur types $\mathfrak{S}_1$,..., $\mathfrak{S}_n$ (resp.
generalized Shur types $\mathfrak{S}_1^{\prime}$,...,
$\mathfrak{S}_n^{\prime}$).

Let $X$ be a scheme and $\mathscr{F}$ be a quasi-coherent sheaf on
$X$. By a {\sl tensor sheaf} of $\mathscr{F}$ we mean a direct sum
$\mathfrak{S}(\mathscr{F})=\mathfrak{S}_1(\mathscr{F})\oplus\cdots\oplus\mathfrak{S}_n(\mathscr{F})$,
for some Shur types $\mathfrak{S}_1$,..., $\mathfrak{S}_n$.
Furthermore, if $\mathscr{F}$ is locally free of finite rank, then
by a {\sl generalized tensor sheaf} of $\mathscr{F}$ we mean a
direct sum
$\mathfrak{S}'(\mathscr{F})=\mathfrak{S}_1^{\prime}(\mathscr{F})\oplus\cdots\oplus\mathfrak{S}_n^{\prime}(\mathscr{F})$,
for some generalized Shur types $\mathfrak{S}_1^{\prime}$,...,
$\mathfrak{S}_n^{\prime}$.
\end{definition}

\vskip 10pt By the above argument, we see that any $\mathfrak{S}$ in
Definition 1 gives a covariant functor
$\mathfrak{M}_R\to\mathfrak{M}_R$, where $\mathfrak{M}_R$ is the
category of $R$-modules; and a covariant functor
$\mathfrak{C}\mathfrak{o}\mathfrak{h}_X\rightarrow\mathfrak{C}\mathfrak{o}\mathfrak{h}_X$,
where $\mathfrak{C}\mathfrak{o}\mathfrak{h}_X$ is the category of
quasi-coherent sheaves over $X$.

\section{CANONICAL ISOMORPHISM }

In this section, we are going to prove a global isomorphism of
tensor sheaves of {\sl Schur type} associated with locally free
sheave~$\mathscr{F}$~of finite rank.

\begin{theorem}
Let~$(X,\mathscr{O}_X)$~be a scheme and~$\mathscr{F}$~be a locally
free~$\mathscr{O}_X-$module sheaf of rank~$n$,~$\mathscr{E}$~is a
direct sum of tensor sheaves of schur type associated
to~$\mathscr{F}$,~$\mathscr{E}=\bigoplus_i\mathscr{E}_i$, with
each~$\mathscr{E}_i$~is of rank~$m_i$~and of degree~$d_i$. Then we
have the following canonical isomorphism of invertible sheaves:
$$(\mbox{det}\mathscr{E})\cong(\mbox{det}\mathscr{F})^{\otimes\sum_i\frac{m_id_i}{n}}.$$
For canonical, we mean if we set~$f:X'\rightarrow\,X$~be a morphism
of schemes, set~$\mathscr{E}'=f^*\mathscr{E}$~and
set~$\mathscr{F}'=f^*\mathscr{F}$, then apply~$f^*$~to the above
isomorphism, we
get$$(\mbox{det}\mathscr{E}')\cong(\mbox{det}\mathscr{F}')^{\otimes\sum_i\frac{m_id_i}{n}}.$$
\end{theorem}

\begin{proof}
Our proof will be divided into several steps. Note
that~$\mathscr{E}=\bigoplus_i\mathscr{E}_i$, using the wedge
functor, we have:
$$\begin{array}{rl}
\mbox{det}\mathscr{E}&\cong(\det(\bigoplus_i\mathscr{E}_i))\cong
\bigoplus_{i_1+i_2+\cdots=d_1+d_2+\cdots}
\wedge^{i_1}(\mathscr{E}_1)\otimes_{\mathscr{O}_X}
\wedge^{i_2}(\mathscr{E}_2)\otimes\cdots\\
&\cong\wedge^{m_1}(\mathscr{E}_1)\otimes_{\mathscr{O}_X}
\wedge^{m_2}(\mathscr{E}_2)\otimes_{\mathscr{O}_X}\cdots\\
&\cong\mbox{det}(\mathscr{E}_1)\otimes_{\mathscr{O}_X}
\mbox{det}(\mathscr{E}_2)\otimes_{\mathscr{O}_X}\cdots\\
\end{array}$$
Hence we only need to prove for each~$i$, the following equality
holds:$$\mbox{det}\mathscr{E}_i\cong(\mbox{det}\mathscr{F})^{\otimes\frac{m_id_i}{n}}.$$
So we may reduce the question to the case where~$\mathscr{E}$~is a
tensor sheaf of schur type of rank~$m$~and of degree~$d$~associated
to free~$\mathscr{O}_X-$sheaf~$\mathscr{F}$~of rank~$n$.

setp1. First assume~$\frac{md}{n}$~is an integer, We will prove this
fact in the end of this step.
Then~$\mbox{det}\mathscr{E}$~and~$(\mbox{det}\mathscr{F})^{\otimes\frac{md}{n}}$~are
all invertible sheaves, so locally checking, there always admits an
isomorphism, the main problem is how to glue these local data to get
a global isomorphism. Our method is to prove this isomorphism is
canonical. Take an affine open
neighborhood~$U=\mbox{Spec}(R)$~of~$X$,
then~$\mathscr{F}|_U\cong\,\widetilde{M}$, where~$M$~is
free~$R-$module of
rank~$n$~and~$\mathscr{E}|_U\cong\widetilde{M'}$~is a tensor module
of schur type associated to~$M$. Suppose~$M$~admits a
basis~${m_i}$~such that~$\varphi\,m_i=\lambda_im_i$~for
some~$\varphi\in\,GL(n,R)$~and~$\lambda_i\in\,R$~, then according to
our definition of tensor modules of schur type, which is in fact the
functorial property of these operations(symmetric products, for
example), ~$M'$~has a canonical basis which is induced from~${m_i}$,
say~${m'_i}$, and~$\varphi$~induces an automorphism~$\varphi'$~acts
on~$M'$, hence~$\varphi'm_i'=\mu_im_i'$~for each~$i$, and
each~$\mu_i$~is a product of~$\lambda_i$s of length~$d$, so the sum
of these eigenvalues is~$dm$(with multiplications). Also note
that~$GL(n,R)$'s diagonal action is also kept to~$M'$, so all the
multiplications of~$\lambda_i$~in~$\mu_i$~should be equal, which
means the multiplication is just~$\frac{md}{n}$.

step2. Using notations as above, ~$S_R^rM$~has a canonical basis
inherited from~$M$,
say~$m_{i_1}\otimes\cdots\otimes\,m_{i_r}$~with~$i_1\leq\,i_2\leq\cdots\leq\,i_r$~and
~$\wedge_R^rM$~has a canonical basis inherited from~$M$, which is
~$m_{i_1}\wedge\cdots\wedge\,m_{i_r}$~with~$i_1<i_2<\cdots<i_r$.
According to our definitions,~$detM'$~have a canonical basis
inherited from given basis of~$M$. Hence we have
as~$(detM)^{\otimes_R\frac{md}{n}}$~and~$detM'$~are free~$R-$modules
of rank 1, we may define an isomorphism from one generator to the
other generator,
$$\begin{array}{cccc}
\phi:&(\mbox{det}M)^{\otimes_R\frac{md}{n}}&\longrightarrow&\mbox{det}M'\\[2mm]
&(m_1\wedge\cdots\wedge\,m_n)^{\otimes\frac{md}{n}}&\mapsto&m_1'\wedge\cdots\wedge\,m_m'
\end{array}.$$
Next we need to prove this is a canonical isomorphism, which
means if we take~$\alpha\in\,GL(n,R)$~,then~$\alpha$~also induces an
automorphism on~$M'$~,say~$\alpha'$~, then there exists the
following commutative diagram:
\[
\begin{CD}
(detM)^{\otimes_R\frac{md}{n}}@>{\phi}>>detM'\\
@VV{(det\alpha)^{\frac{md}{n}}}V@V{det\alpha'}VV\\
(detM)^{\otimes_R\frac{md}{n}}@>{\phi}>>detM'
\end{CD}
\]
So we only need to check~$det\alpha'=(det\alpha)^{\frac{md}{n}}$~for
all~$\alpha\in\,GL(n,R)$~.Note that according to step1,
if~$\alpha$~admits a basis such that~$\alpha$~is diagonal, then the
equality already holds!

As an element~$\alpha=(\alpha_{ij})\in\,GL(n,R)$~corresponds to an
algebraic homomorphism
$(\mathbb{Z}[x_{ij}]_{1\leq\,i,j\leq\,n})_{det(x_{ij})}\rightarrow\,R$~,
which is defined by~$\alpha_{ij}\mapsto\,x_{ij}$. In fact this is a
morphism from~$\mbox{Spec}(R)$~to
$\mbox{Spec}((\mathbb{Z}[x_{ij}]_{1\leq\,i,j\leq\,n})_{det(x_{ij})})$~,
so in order to check the equality over~$R$~, we only need to check
it
over~$X=Spec((\mathbb{Z}[x_{ij}]_{1\leq\,i,j\leq\,n})_{det(x_{ij})})$.

step3. As~$det\alpha'=(det\alpha)^{\frac{md}{n}}$~defines a closed
subset of~$X$, say~$V$. We are going to prove it contains a nonempty
open subset~$U$~, hence the taking closure we
have~$\bar{U}\subset\bar{V}\subset\bar{V}$. As~$X$~is irreducible,
every nonempty open subset is dense, hence we have~$V=X$.

Set~$U=\{\varphi|\ \chi_\varphi\ \mbox{has n different
eigenvalues}\}$, according to step1, we have~$U\subset\,V$. Also
note that~$\chi_\varphi$~has~$n$~different eigenvalues if and only
if~$\triangle_\varphi\neq0$, so~$U$~is an open subset of~$X$.
Set~$R_0=(\mathbb{Z}[x_{ij}]_{1\leq\,i,j\leq\,n})_{det(x_{ij})}$,~$K=\mbox{q.f.}(R_0)$~and~$\varphi'=(x_{ij})$.
Then~$\chi_{\varphi'}=|\lambda\,I-(x_{ij})|$~be the characteristic
polynomial of~$(x_{ij})$, and~$L/K$~be the splitting filed
of~$\chi_{\varphi'}$~over~$K$, then~$\varphi'$~has~$n$~different
eigenvalues~$\lambda_i\in\,L$~such
that~$\varphi'v_i=\lambda_iv_i$~for~$v_i\in\,L^n$~for
all~$i=1,\cdots\,,n$. We may take a common denominator~$f$~such that
these equalities hold in a open neighborhood of the generic point
of~$R_0$, hence~$U\neq\emptyset$, taking closure we have~$V=X$.

As the isomorphism is canonical, hence we may glue together to get a
global isomorphism.
\end{proof}

\section{Canonical Homomorphisms}

In this section,we are going to study canonical homomorphisms of
tensor sheaves. More precisely,we will construct canonical
homomorphisms for~$\mathscr{O}_X$-module~$\mathscr{F}$. In order to
define homomorphisms of tensor sheaves of~$\mathscr{F}$, note that
if~$\mathscr{F}$~is locally free of rank~$n$, we can use local
generators of~$\mathscr{F}(U)$~for open subset~$U\subset\,X$. The
reason is that for locally free sheaves, when we take local
generators, any two different choices of generators are differed by
a~$n\times\,n$~invertible matrix, and the gluing data is also
reflected by~$n\times\,n$~invertible matrix, so in order to check
the morphism is well defined, we just need to fix a basis, and prove
the morphism is independent of the basis we choose before. But for
general~$\mathscr{O}_X-$module~$\mathscr{F}$, we don't know any
information of local generators and information of gluing data. But
for general tensor sheaves, it admits a canonical permutation
group's action, this is the point we use in this paper. In fact the
main idea is already occurred in the de Rham Complex in Prof.
Ke-Zheng Li's lectures in "Moduli Space And Its Applications". For
details, see \cite{L3}.

Let~$k$~and~$n$~be two integers,~$(X,\mathscr{O}_X)$~be a scheme and
~$\mathscr{F}$~be an~$\mathscr{O}_X-$module.

\begin{theorem}
There exists canonical homomorphism$$\varphi_{n,k}:\
S^n_{\mathscr{O}_X}(S^k_{\mathscr{O}_X}\mathscr{F})
\longrightarrow\,S^k_{\mathscr{O}_X}(\wedge^n_{\mathscr{O}_X}\mathscr{F}).$$if~$k$~is
an even integer.

There exists canonical homomorphism
$$\phi_{k,n}:\ S^k_{\mathscr{O}_X}(\wedge^n_{\mathscr{O}_X}\mathscr{F})\longrightarrow
S^n_{\mathscr{O}_X}(S^k_{\mathscr{O}_X}\mathscr{F}).$$
\end{theorem}

\begin{proof}
As~$S^k_{\mathscr{O}_X}(\wedge^n_{\mathscr{O}_X}\mathscr{F})$~and~$S^n_{\mathscr{O}_X}(S^k_{\mathscr{O}_X}\mathscr{F})$~are
all quotient sheaves of~$\mathscr{F}^{\otimes(n+k)}$, so in order to
define~$\varphi_{n,k}$~and~$\phi_{k,n}$, we only need to define
endomorphisms of~$\mathscr{F}^{\otimes(n+k)}$~such
that~$\varphi_{n,k}$~and~$\phi_{k,n}$~can be induced from these
endomorphisms.

Note~$\mathscr{F}^{\otimes(n+k)}$~admits a canonical permutation
group~$\mathfrak{S}_{n+k}$'s action, so in order to define global
homomorphisms, we could use these permutations.
Take~$\sigma_2,\cdots,$
$\sigma_n\in\mathfrak{S}_k$,
define~$\tilde{\varphi}_{n,k}$~as follows:
$$\begin{array}{ccc}
\mathscr{F}^{\otimes(n+k)}&\longrightarrow&\mathscr{F}^{\otimes(n+k)}\\
\begin{split}(a_{11}\otimes\cdots\otimes\,a_{1k})\otimes\cdots\otimes\\
(a_{n1}\otimes\cdots\otimes\,a_{nk})\end{split}&\mapsto&
\begin{split}\sum_{\sigma_2,\cdots,\sigma_n\in\mathfrak{S}_k}
(a_{11}\otimes\,a_{2\sigma_2(1)}\otimes\cdots\otimes\,a_{n\sigma_n(1)})\otimes\\
(a_{12}\otimes\,a_{2\sigma_2(2)}\otimes\cdots\otimes\,a_{n\sigma_n(2)})\otimes\cdots\\
\otimes(a_{1k}\otimes\,a_{2\sigma_2(k)}\otimes\cdots\otimes\,a_{n\sigma_n(k)})
\end{split}
\end{array}$$
Or equivalently, we may write this homomorphism
as$$\otimes_{i=1}^{n}(\otimes_{j=1}^ka_{ij})\mapsto
\sum\limits_{\sigma_2,\cdots,\sigma_n\in\mathfrak{S}_k}\otimes_{i=1}^k(a_{1i}\otimes(\otimes_{j=2}^na_{j\sigma_j(i)})).$$

It is not difficult to verify on ideal of definitions for these 2
tensor sheaves that~$\tilde{\varphi}_{n,k}$~induces a homomorphism
from~$S^n_{\mathscr{O}_X}(S^k_{\mathscr{O}_X}\mathscr{F})$~to~$S^k_{\mathscr{O}_X}(\wedge^n_{\mathscr{O}_X}\mathscr{F})$.
We denote it by~$\varphi_{n,k}$. If we write  this homomorphism in
local coordinates, it is
~$\otimes_{i=1}^{n}(\otimes_{j=1}^ka_{ij})\mapsto
\sum\limits_{\sigma_2,\cdots,\sigma_n\in\mathfrak{S}_k}\prod\limits_{i=1}^k(a_{1i}\wedge(\wedge_{j=2}^na_{j\sigma_j(i)}))$.

In order to define~$\phi_{k,n}$, we still use the
fact~$\mathscr{F}^{\otimes(n+k)}$~have permutation group's canonical
action. Take~$\tau_2,\cdots,\tau_k\in\mathfrak{S}_n$, we
define~$\tilde{\phi}_{k,n}$~as follows:
$$\begin{array}{ccc}
\mathscr{F}^{\otimes(n+k)}&\longrightarrow&\mathscr{F}^{\otimes(n+k)}\\
\begin{split}(a_{11}\otimes\cdots\otimes\,a_{1n})\otimes\cdots\otimes\\
(a_{k1}\otimes\cdots\otimes\,a_{kn})\end{split}&\mapsto&
\begin{split}\sum_{\tau_2,\cdots,\tau_k\in\mathfrak{S}_n}(-1)^{\tau(\tau_2\cdots\tau_k)}
(a_{11}\otimes\,a_{2\tau_2(1)}\otimes\cdots\otimes\,a_{k\tau_k(1)})\\
\otimes(a_{12}\otimes\,a_{2\tau_2(2)}\otimes\cdots\otimes\,a_{k\tau_k(2)})\otimes\\
\cdots\otimes(a_{1n}\otimes\,a_{2\tau_2(n)}\otimes\cdots\otimes\,a_{k\tau_k(n)})
\end{split}
\end{array}
$$
Or equivalently we may write this homomorphism as
$$\otimes_{i=1}^{k}(\otimes_{j=1}^na_{ij})\mapsto
\sum\limits_{\tau_2,\cdots,\tau_n\in\mathfrak{S}_n}(-1)^{\tau(\tau_2\cdots\tau_k)}
\otimes_{i=1}^n(a_{1i}\otimes(\otimes_{j=2}^ka_{j\sigma_j(i)})).$$

It is easy to verify that~$\tilde{\phi}_{k,n}$~induces a
homomorphism
from~$S^k_{\mathscr{O}_X}(\wedge^n_{\mathscr{O}_X}\mathscr{F})$~to
~$S^n_{\mathscr{O}_X}(S^k_{\mathscr{O}_X}\mathscr{F})$, which means
we have the following commutative diagram:
$$\xymatrix{
\mathscr{F}^{\otimes(k+n)}\ar[r]^{\tilde{\phi}_{k,n}}\ar[d]&\mathscr{F}^{\otimes(k+n)}\ar[d]\\
S^k_{\mathscr{O}_X}(\wedge^n_{\mathscr{O}_X}\mathscr{F})\ar@{.>}[r]&
S^n_{\mathscr{O}_X}(S^k_{\mathscr{O}_X}\mathscr{F})\\}$$ We denote
this homomorphism by~$\phi_{k,n}$. If we write this homomorphism in
local coordinates, it is~$\prod_{i=1}^k(\wedge_{j=1}^na_{ij})\mapsto
\sum\limits_{\tau_2,\cdots,\tau_k\in\mathfrak{S}_n}((-1)^{\tau_2\cdots\tau_k}\otimes_{i=1}^n(a_{1i}\otimes(\otimes_{j=2}^ka_{j\tau_j(i)})))$.

This finishes the proof.
\end{proof}

But unfortunately for most cases we can not get the composition
properties for~$\varphi_{n,k}\circ\phi_{k,n}$~and
$\phi_{k,n}\circ\varphi_{n,k}$, even when~$\mathscr{F}$~is locally
free. The main problem is although~$\varphi_{n,k}$~and~$\phi_{k,n}$
are globally well defined, it still has too many additional
combinational terms if we consider compositions for these
homomorphisms. A special case
is~$\varphi_{2,k}\circ\phi_{k,2}$~when~$\mathscr{F}$~is locally
free~$\mathscr{O}_X-$module of rank 2. In this case, we have
$$\begin{array}{c}
\phi_{k,2}((x\wedge\,y)^{k})=
\sum\limits_{i=1}^k((-1)^{k-i}\frac{(k-1)!}{(k-i)!(i-1)!}x^i\otimes\,y^{k-i}
\cdot\,x^{k-i}\otimes\,y^i).\\[3mm]
\varphi_{2,k}(x^i\otimes\,y^{k-i}\cdot\,x^{k-i}\otimes\,y^i)=(-1)^{k-i}(k-i)!\,i!(x\wedge\,y)^{k}.
\end{array}
$$
So the composition is a scalar multiplication
by~$\sum\limits_{i=1}^k(-1)^{k-i}\frac{(k-1)!}{(k-i)!(i-1)!}(-1)^{k-i}(k-i)!\,i!=
\sum\limits_{i=1}^ki\cdot(k-1)!=\frac{(k+1)!}{2}$. So we
have~$\varphi_{2,k}\circ\phi_{k,2}=
\frac{(k+1)!}{2}\cdot\,id_{S_{\mathscr{O}_X}^k(\wedge_{\mathscr{O}_X}^2\mathscr{F})}$.
We also compute the
case~$\varphi_{3,2}\circ\phi_{2,3}$~when~$\mathscr{F}$~is locally
free of rank 3, it is a scalar multiplication by 12, which is equal
to~$\frac{(2+3-1)!}{2}$. Hence we think the general case is still
true. Namely the
equality$$\varphi_{n,k}\circ\phi_{k,n}=\frac{(k+n-1)!}{2}\cdot\,
id_{S_{\mathscr{O}_X}^k(\wedge_{\mathscr{O}_X}^n\mathscr{F})}\
.$$if~$\mathscr{F}$~is locally free of rank~$n$.

Next we are going to use this theorem to get canonical homomorphisms
of $S^2_{\mathscr{O}_X}(S^n_{\mathscr{O}_X}\mathscr{F})$,
where~$\mathscr{F}$~is an~$\mathscr{O}_X-$module. Our aim to study
this kind of sheaf comes from \cite{FH}, as there exists a
decomposition of this kind of module. But in there, they deal with
representation theory, the base is $\mathbb{C}$, we want to get more
general result over general base~$X$~and the main idea is directed
by \cite{FH}.

First we will state the two steps handled in \cite{FH}
if~$\mathscr{F}$~is a locally free~$\mathscr{O}_X-$module of rank 2.
There exists a canonical homomorphism
from~$(\wedge_{\mathscr{O}_X}^2\mathscr{F})^{\otimes2}\otimes
S_{\mathscr{O}_X}^2(S_{\mathscr{O}_X}^{n-2}\mathscr{F})$ to
~$S_{\mathscr{O}_X}^2(S_{\mathscr{O}_X}^n\mathscr{F})$~which makes
the kernel for the upper canonical surjective homomorphism. If so we
could get the short exact
sequence:~$$0\rightarrow(\wedge_{\mathscr{O}_X}^2\mathscr{F})^{\otimes2}
\otimes\,S_{\mathscr{O}_X}^2(S_{\mathscr{O}_X}^{n-2}\mathscr{F})
\longrightarrow S_{\mathscr{O}_X}^2(S_{\mathscr{O}_X}^n\mathscr{F})
\longrightarrow S_{\mathscr{O}_X}^{2n}\mathscr{F}\rightarrow0.$$
Then the problem is whether we can construct a canonical
homomorphism from
$S_{\mathscr{O}_X}^2(S_{\mathscr{O}_X}^n\mathscr{F})$~to
$(\wedge_{\mathscr{O}_X}^2\mathscr{F})^{\otimes2}\otimes
S_{\mathscr{O}_X}^2(S_{\mathscr{O}_X}^{n-2}\mathscr{F})$ which forms
a canonical retraction for the upper exact sequence. Of course this
canonical retraction should be constructed from the following
commutative diagram:
$$\xymatrix{
\mathscr{F}^{\otimes2n}\ar[r]^{unknown}\ar[d]&\mathscr{F}^{\otimes2n}\ar[d]\\
S_{\mathscr{O}_X}^2(S_{\mathscr{O}_X}^n\mathscr{F})\ar@{.>}[r]
&(\wedge_{\mathscr{O}_X}^2\mathscr{F})^{\otimes2}\otimes\,S_{\mathscr{O}_X}^2
(S_{\mathscr{O}_X}^{n-2}\mathscr{F}).}$$ If these 2 steps are
correct, we may
write~$S_{\mathscr{O}_X}^2(S_{\mathscr{O}_X}^n\mathscr{F})$~as a
direct sum by induction. Hence we may decompose the locally free
sheaf~$S_{\mathscr{O}_X}^2(S_{\mathscr{O}_X}^{n-2}\mathscr{F})$.

But unfortunately what we have done now is not as above.  We can
construct global homomorphisms. The main problem is although these
homomorphisms are globally well defined, as before they have too
many combinational terms. Hence for composition they don't have good
properties, even when~$\mathscr{F}$~is locally
free~$\mathscr{O}_X-$module of rank 2.

Let~$n\geq3$~be an integer.
\begin{theorem}
There exists the following canonical homomorphisms
$$\xymatrix{S^2_{\mathscr{O}_X}(\wedge^2_{\mathscr{O}_X}\mathscr{F})
\otimes\,S^2_{\mathscr{O}_X}(S_{\mathscr{O}_X}^{n-2}\mathscr{F})\ar@{>}[r]^{\quad\qquad
i} &S^2_{\mathscr{O}_X}(S^n_{\mathscr{O}_X}\mathscr{F})\ar[r]^q
&S^{2n}_{\mathscr{O}_X}\mathscr{F}}$$ and
$$\xymatrix{S^2_{\mathscr{O}_X}(\wedge^2_{\mathscr{O}_X}\mathscr{F})
\otimes\,S^2_{\mathscr{O}_X}(S_{\mathscr{O}_X}^{n-2}\mathscr{F})
&S^2_{\mathscr{O}_X}(S^n_{\mathscr{O}_X}\mathscr{F})\ar[l]_{\qquad\quad
j} &S^{2n}_{\mathscr{O}_X}\mathscr{F}\ar[l]_{\varphi}}$$such
that~$q\circ\,i=0$~,~$q\circ\varphi$~is a scalar multiplication
of~$S^2_{\mathscr{O}_X}\mathscr{F}$~by~$\frac{(2n-1)\cdots(n+1)}{(n-1)!}$~.
\end{theorem}

\begin{proof}
It is easy to see there exists a canonical surjective
homomorphism~$$S_{\mathscr{O}_X}^2(S_{\mathscr{O}_X}^n\mathscr{F})
\longrightarrow\,S_{\mathscr{O}_X}^{2n}\mathscr{F}\rightarrow0.$$which
is induced by the identity map
on~$\mathscr{F}^{\otimes2n}\rightarrow\mathscr{F}^{\otimes2n}$. We
denote this quotient homomorphism by~$q$.

Next set~$N$~be a subset of the set~$\{1,2,\cdots,2n\}$~such
that~$1\in\,N$~and~$|N|=n$. Define a
homomorphism~$\tilde{\varphi}$~as follows:
$$\begin{array}{cccc}
\tilde{\varphi}:&\mathscr{F}^{\otimes(2n)}&\longrightarrow&\mathscr{F}^{\otimes(2n)}\\
&a_1\otimes\cdots\otimes\,a_{2n}&\mapsto&
\sum\limits_N((\otimes_{i\in\,N}a_i)\otimes(\otimes_{j\bar{\in}N}a_j))\ .\\
\end{array}$$
It is not difficult to check that~$\tilde{\varphi}$~induces a
homomorphism
from~$S_{\mathscr{O}_X}^{2n}\mathscr{F}$~to~$S_{\mathscr{O}_X}^2(S_{\mathscr{O}_X}^n\mathscr{F})$.
Namely we have the following commutative diagram:
$$\xymatrix{\mathscr{F}^{\otimes(2n)\ar[r]^{\tilde{\varphi}}}\ar[d]&\mathscr{F}^{\otimes(2n)}\ar[d]\\
S_{\mathscr{O}_X}^{2n}\mathscr{F}\ar@{.>}[r]&S_{\mathscr{O}_X}^2(S_{\mathscr{O}_X}^n\mathscr{F})\\}$$
We denote this homomorphism by~$\varphi$.

For~$q\circ\varphi$, note that for a general
term~$\alpha\in\,S_{\mathscr{O}_X}^{2n}\mathscr{F}$,
$\varphi(\alpha)$~is just proper permutations of coordinates
of~$\alpha$, it does not have any other terms, so if mapped
into~$S_{\mathscr{O}_X}^{2n}\mathscr{F}$~again, they are all the
same element. Hence we only need to compute the multiplication.
According to the definition of the subset~$N$, we
have~$q\circ\varphi=\frac{(2n-1)\cdots(n+1)}{(n-1)!}\cdot\,id_{S_{\mathscr{O}_X}^{2n}\mathscr{F}}$.

Next we will define~$i$~and~$j$. First note that according to the
theorem before, we have canonical homomorphisms
$$\begin{array}{cccc}
\phi_{2,2}:&S^2_{\mathscr{O}_X}(\wedge^2_{\mathscr{O}_X}\mathscr{F})&\longrightarrow
&S^2_{\mathscr{O}_X}(S^2_{\mathscr{O}_X}\mathscr{F})\\
&(a_{11}\wedge\,a_{12})\cdot(a_{21}\wedge\,a_{22})&\mapsto&
\begin{split}(a_{11}\otimes\,a_{21})\cdot(a_{12}\otimes\,a_{22})-\\
(a_{11}\otimes\,a_{22})\cdot(a_{12}\otimes\,a_{21})
\end{split}\ .
\end{array}$$
and
$$\begin{array}{cccc}
\varphi_{2,2}:&S^2_{\mathscr{O}_X}(S^2_{\mathscr{O}_X}\mathscr{F})&\longrightarrow
&S^2_{\mathscr{O}_X}(\wedge^2_{\mathscr{O}_X}\mathscr{F})\\
&(a_{11}\otimes\,a_{12})\cdot(a_{21}\otimes\,a_{22})&\mapsto
&\begin{split}
(a_{11}\wedge\,a_{21})\cdot(a_{12}\wedge\,a_{22})+\\
(a_{11}\wedge\,a_{22})\cdot(a_{12}\wedge\,a_{21})
\end{split}\ .
\end{array}$$
Where~$a_{ij}$~are local coordinates of~$\mathscr{F}$~for~$i,j=1,2$.

So in order to define~$i$~and~$j$, we only need to define
homomorphisms$$f:\quad
S^2_{\mathscr{O}_X}(S^2_{\mathscr{O}_X}\mathscr{F})
\otimes\,S^2_{\mathscr{O}_X}(S_{\mathscr{O}_X}^{n-2}\mathscr{F})\rightarrow
S^2_{\mathscr{O}_X}(S^n_{\mathscr{O}_X}\mathscr{F}).$$
and$$g:\quad
S^2_{\mathscr{O}_X}(S^n_{\mathscr{O}_X}\mathscr{F})\rightarrow\,S^2_{\mathscr{O}_X}(S^2_{\mathscr{O}_X}\mathscr{F})
\otimes\,S^2_{\mathscr{O}_X}(S_{\mathscr{O}_X}^{n-2}\mathscr{F}).$$
set~$i=f\circ(\phi_{2,2}\otimes\,id_{S^2_{\mathscr{O}_X}(S_{\mathscr{O}_X}^{n-2}\mathscr{F})})$~and
$j=(\varphi_{2,2}\otimes\,id_{S^2_{\mathscr{O}_X}(S_{\mathscr{O}_X}^{n-2}\mathscr{F})})\circ\,g$.
In order to check the composition properties, we just need to check
it on local generators.

Define~$\tilde{f}$~and~$\tilde{g}$~as follows:
$$\begin{array}{cccc}
\tilde{f}:&\mathscr{F}^{\otimes(2n)}&\longrightarrow&\mathscr{F}^{\otimes(2n)}\\
&\begin{split}
(t_1\otimes\,t_2\otimes\,t_3\otimes\,t_4)\otimes((a_1\otimes\cdots
\\\otimes\,a_{n-2})\cdot(b_1\otimes\cdots\otimes\,b_{n-2}))\end{split}
&\mapsto&\begin{split}
(t_1t_3a_1\cdots\,a_{n-2})\otimes(t_2t_4b_1\cdots\,b_{n-2})\\
+(t_1t_4a_1\cdots\,a_{n-2})\otimes(t_2t_3b_1\cdots\,b_{n-2})\\
+(t_2t_4a_1\cdots\,a_{n-2})\otimes(t_1t_3b_1\cdots\,b_{n-2})\\
+(t_2t_3a_1\cdots\,a_{n-2})\otimes(t_1t_4b_1\cdots\,b_{n-2})
\end{split}\end{array}$$
$$\begin{array}{cccc}
\tilde{g}:&\mathscr{F}^{\otimes(2n)}&\longrightarrow&\mathscr{F}^{\otimes(2n)}\\
&\begin{split}(a_1\cdots\,a_n)\otimes\\
(b_1\cdots\,b_n)\end{split}&\mapsto&
\begin{split}\sum_{i<j}\sum_{k<l}(a_i\otimes\,b_k\cdot\,a_j\otimes\,b_l+
a_i\otimes\,b_l\cdot\,a_j\otimes\,b_k)\cdot\\
(a_1\cdots\widehat{a_i}\cdots\widehat{a_j}\cdots\,a_n)\otimes(b_1\cdots\widehat{b_k}\cdots\widehat{b_l}\cdots\,b_n)
\end{split}\end{array}$$
It is not difficult to check on ideal of definitions
that~$\tilde{f}$~and~$\tilde{g}$~induce homomorphisms~$f$~and~$g$.
In order to check the computation properties we write~$i$~and~$j$~in
local coordinates. We have
$$\begin{array}{l}
\quad
i((x_1\wedge\,y_1\cdot\,x_2\wedge\,y_2)\cdot(a_1\otimes\cdots\otimes\,a_{n-2}\cdot\,b_1\otimes\cdots\otimes\,b_{n-2}))\\[1mm]
=(x_1y_2a_1\cdots\,a_{n-2})\otimes(x_2y_1b_1\cdots\,b_{n-2})+(x_2y_1a_1\cdots\,a_{n-2})\otimes(x_1y_2b_1\cdots\,b_{n-2})-\\[1mm]
\quad(x_1x_2a_1\cdots\,a_{n-2})\otimes(y_1y_2b_1\cdots\,b_{n-2})-(y_1y_2a_1\cdots\,a_{n-2})\otimes(x_1x_2b_1\cdots\,b_{n-2})
\end{array}$$
And
$$\begin{array}{l}
\quad j\,((a_1\otimes\cdots\otimes\,a_n)\cdot(b_1\otimes\cdots\otimes\,b_n))\\[1mm]
=\sum\limits_{i<j}\,\sum\limits_{k<l}-(a_i\wedge\,b_l\cdot\,a_j\wedge\,b_k+
a_i\wedge\,b_k\cdot\,a_j\otimes\,b_l)\cdot\\[1mm]
\quad(a_1\cdots\widehat{a_i}\cdots\widehat{a_j}\cdots\,a_n)\cdot(b_1\cdots\widehat{b_k}\cdots\widehat{b_l}\cdots\,b_n)
\end{array}$$
Hence we define all the homomorphisms. It is not difficult to see
that for a general
element~$\alpha\in\,S^2_{\mathscr{O}_X}(\wedge^2_{\mathscr{O}_X}\mathscr{F})
\otimes\,S^2_{\mathscr{O}_X}(S_{\mathscr{O}_X}^{n-2}\mathscr{F})$,~$i(\alpha)$~is
just permutations of coordinates of~$\alpha$,when mapped
into~$S_{\mathscr{O}_X}^{2n}\mathscr{F}$, they are the same element.
Also note that these terms have different signs,
hence~$q\circ\,i=0$.

This finishes the proof.
\end{proof}

\section{Canonical Decompositions}

In this final section, we are going to give 2 concrete examples of
the theorems in section2.

First we are going to prove the following lemma, which is an
exercise in \cite{H}, which generalized the exercise in\cite{L2}.

\begin{lemma}
Let ~$(X,\mathscr{O}_X)$~be a scheme,
suppose~$0\rightarrow\mathscr{F}'\rightarrow\mathscr{F}\rightarrow\mathscr{F}''\rightarrow0$~is
an~$\mathscr{O}_X-$module's exact
sequence,~$\mathscr{F}'$,~$\mathscr{F}$~and~$\mathscr{F}''$~are
locally free of rank~$m$,~$m+n$~and~$n$~. Then we
have\begin{equation}
\wedge_{\mathscr{O}_X}^{m+n}\mathscr{F}\cong\wedge_{\mathscr{O}_X}^m\mathscr{F}'
\otimes_{\mathscr{O}_X}\wedge_{\mathscr{O}_X}^n\mathscr{F}''.
\end{equation}
\end{lemma}

\begin{proof}
It is easy to see
that~$\wedge_{\mathscr{O}_X}^{m+n}\mathscr{F}$~and~$\wedge_{\mathscr{O}_X}^m\mathscr{F}'
\otimes_{\mathscr{O}_X}\wedge_{\mathscr{O}_X}^n\mathscr{F}''$~are
both locally free~$R-$modules of rank 1. Use the functorial property
of~$\wedge$~and the exact sequence
of~$0\rightarrow\mathscr{F}'\rightarrow\mathscr{F}\rightarrow\mathscr{F}''\rightarrow0$~we
may get canonical surjective
homomorphism~$\wedge_{\mathscr{O}_X}^n\mathscr{F}\rightarrow\wedge_{\mathscr{O}_X}^n
\mathscr{F}''\rightarrow0$~,~$\wedge_{\mathscr{O}_X}^m\mathscr{F}'\otimes_{\mathscr{O}_X}-$~for
the upper sequence and taking the kernel we get exact sequence
\begin{equation}
0\rightarrow\mbox{Ker}(\alpha)\rightarrow\wedge_{\mathscr{O}_X}^m\mathscr{F}'
\otimes_{\mathscr{O}_X}\wedge_{\mathscr{O}_X}^n\mathscr{F}
\xrightarrow{\alpha}\wedge_{\mathscr{O}_X}^m\mathscr{F}'
\otimes_{\mathscr{O}_X}\wedge_{\mathscr{O}_X}^n\mathscr{F}''\rightarrow0.
\end{equation}
From the definition it is easy to see that~$\mbox{Ker}(\alpha)$~is a
subsheaf of~$\wedge_{\mathscr{O}_X}^m\mathscr{F}'
\otimes_{\mathscr{O}_X}\wedge_{\mathscr{O}_X}^n\mathscr{F}$~which
are locally generated by sections of the
form~$(a_1'\wedge\cdots\wedge\,a_m')\otimes(a_1\wedge\cdots\wedge\,a_n)$~with~$a_i\in\,M'$~for
some~$i=1,\cdots,n$. Also note that there exists a canonical
homomorphism from
~$\wedge_{\mathscr{O}_X}^m\mathscr{F}\otimes_{\mathscr{O}_X}
\wedge_{\mathscr{O}_X}^n\mathscr{F}\rightarrow
\wedge_{\mathscr{O}_X}^{m+n}\mathscr{F}$~by
mapping~$(a_1'\wedge\cdots\wedge\,a_m',a_1\wedge\cdots\wedge\,a_n)$~to~$a_1'\wedge\cdots
\wedge\,a_m'\wedge\,a_1\wedge\cdots\wedge\,a_n$, in fact one only
need to compute there exists the following commutative diagram:
$$\xymatrix{ \mathscr{F}^{\otimes_{\mathscr{O}_X}(m+n)}
\ar[r]^{id}\ar[d]&\mathscr{F}^{\otimes_{\mathscr{O}_X}(m+n)}\ar[d]\\
\wedge_{\mathscr{O}_X}^m\mathscr{F}\otimes\wedge_{\mathscr{O}_X}^n\mathscr{F}
\ar@{.>}[r]&\wedge_{\mathscr{O}_X}^{m+n}\mathscr{F}.}$$ Also we have
a canonical
homomorphism~$\wedge_{\mathscr{O}_X}^m\mathscr{F}'\otimes_{\mathscr{O}_X}
\wedge_{\mathscr{O}_X}^n\mathscr{F}\rightarrow\wedge_{\mathscr{O}_X}^m\mathscr{F}
\otimes_{\mathscr{O}_X}\wedge_{\mathscr{O}_X}^n\mathscr{F}$~. Taking
composition we
get~$$\wedge_{\mathscr{O}_X}^m\mathscr{F}'\otimes_{\mathscr{O}_X}
\wedge_{\mathscr{O}_X}^n\mathscr{F}\rightarrow\wedge_{\mathscr{O}_X}^{m+n}\mathscr{F}.$$Denote
this morphism by~$\beta$~. Taking the kernel we have
\begin{equation}
0\rightarrow\mbox{Ker}(\beta)\rightarrow
\wedge_{\mathscr{O}_X}^m\mathscr{F}'\otimes_{\mathscr{O}_X}
\wedge_{\mathscr{O}_X}^n\mathscr{F}\rightarrow
\wedge_{\mathscr{O}_X}^{m+n}\mathscr{F}.\end{equation}From the
construction it is easy to see that~$\mbox{Ker}(\beta)$~is subsheaf
of~$\wedge_{\mathscr{O}_X}^m\mathscr{F}'\otimes_{\mathscr{O}_X}
\wedge_{\mathscr{O}_X}^n\mathscr{F}$~locally generated by sections
of the
form~$(a_1'\wedge\cdots\wedge\,a_m')\otimes(a_1\wedge\cdots\wedge\,a_n)$~with
some~$a_i'=a_j$~for some pair~$i\,,j$~. Then it is easy to see
that~$\mbox{Ker}(\beta)$~is a submodule of~$\mbox{Ker}(\alpha)$~. So
we may get the following commutative diagram
$$\begin{array}{ccccccccc}
0&\rightarrow&\mbox{Ker}(\beta)&\rightarrow&
\wedge_{\mathscr{O}_X}^m\mathscr{F}'\otimes_{\mathscr{O}_X}
\wedge_{\mathscr{O}_X}^n\mathscr{F}&\xrightarrow{\beta}&(\wedge_{\mathscr{O}_X}^m\mathscr{F}'\otimes_{\mathscr{O}_X}
\wedge_{\mathscr{O}_X}^n\mathscr{F})/\mbox{Ker}(\beta)&\rightarrow&0\\
&&\downarrow&&\downarrow&&\downarrow&&\\
0&\rightarrow&\mbox{Ker}(\alpha)&\rightarrow&\wedge_{\mathscr{O}_X}^m\mathscr{F}'
\otimes_{\mathscr{O}_X}\wedge_{\mathscr{O}_X}^n\mathscr{F}
&\xrightarrow{\alpha}&\wedge_{\mathscr{O}_X}^m\mathscr{F}'
\otimes_{\mathscr{O}_X}\wedge_{\mathscr{O}_X}^n\mathscr{F}''&\rightarrow&0\\
\end{array}$$According
the snake lemma we get a canonical surjective homomorphism from
~$(\wedge_{\mathscr{O}_X}^m\mathscr{F}'\otimes_{\mathscr{O}_X}
\wedge_{\mathscr{O}_X}^n\mathscr{F})/\mbox{Ker}(\beta)$ to
$\wedge_{\mathscr{O}_X}^m\mathscr{F}'
\otimes_{\mathscr{O}_X}\wedge_{\mathscr{O}_X}^n\mathscr{F}''$~.
 So~$(\wedge_{\mathscr{O}_X}^m\mathscr{F}'\otimes_{\mathscr{O}_X}
\wedge_{\mathscr{O}_X}^n\mathscr{F})/\mbox{Ker}(\beta)$~is locally
free~$\mathscr{O}_X-$module of rank~$\geq1$. But in
fact~$(\wedge_{\mathscr{O}_X}^m\mathscr{F}'\otimes_{\mathscr{O}_X}
\wedge_{\mathscr{O}_X}^n\mathscr{F})/\mbox{Ker}(\beta)$~is a
submodule of~$\wedge^{m+n}_{\mathscr{O}_X}\mathscr{F}$~, which
is~$\mathscr{O}_X-$module locally free of rank 1. Thus we
have~$(\wedge_{\mathscr{O}_X}^m\mathscr{F}'\otimes_{\mathscr{O}_X}
\wedge_{\mathscr{O}_X}^n\mathscr{F})/\mbox{Ker}(\beta)=\wedge^{m+n}_{\mathscr{O}_X}\mathscr{F}$~and
a surjective homomorphism from the former to the latter. Thus this
is an isomorphism. This finishes the proof.
\end{proof}

\textbf{Remark:}In here we cannot use the formula of Theorem1.17. As
in here~$\mathscr{F}'$,
$\mathscr{F}$~and~$\mathscr{F}''$~are %
locally free~$\mathscr{O}_X-$modules, so there doesn't exist a
canonical split exact sequence.

Consider the special case where~$\mathscr{F}$~is a locally
free~$\mathscr{O}_X$-module of rank 2,we have the following theorem.
\begin{theorem}
There exists a canonical exact
sequence:~$$0\longrightarrow\,S_{\mathscr{O}_X}^2(\wedge_{\mathscr{O}_X}^2\mathscr{F})
\longrightarrow\,S_{\mathscr{O}_X}^2
(S_{\mathscr{O}_X}^2\mathscr{F})\longrightarrow\,S_{\mathscr{O}_X}^4\mathscr{F}\longrightarrow0.$$
There also admits a canonical homomorphism~$\tau:S_{\mathscr{O}_X}^2
(S_{\mathscr{O}_X}^2\mathscr{F})\longrightarrow\,S_{\mathscr{O}_X}^2(\wedge_{\mathscr{O}_X}^2\mathscr{F})$,
composite~$\tau$~with the canonical inclusion, it it multiplication
by~$3$. So if~$X$~is an~$\mbox{Spec}(\mathbb{Z}[\frac{1}{3}])$
scheme, which means if 3 is an invertible element
in~$\mathscr{O}_X$, the canonical exact sequence splits, hence we
have canonical direct
sum$$S_{\mathscr{O}_X}^2(S_{\mathscr{O}_X}^2\mathscr{F})\cong
S_{\mathscr{O}_X}^2(\wedge_{\mathscr{O}_X}^2\mathscr{F})
\oplus\,S_{\mathscr{O}_X}^4\mathscr{F}.$$
\end{theorem}

\begin{proof}
As in this case~$\mathscr{F}$~is locally free of rank 2, we are
going to prove this theorem on local constructions.

First it's easy to
compute~$S_{\mathscr{O}_X}^2(\wedge_{\mathscr{O}_X}^2\mathscr{F})$ ,
$S_{\mathscr{O}_X}^2 (S_{\mathscr{O}_X}^2\mathscr{F})$
and~$S_{\mathscr{O}_X}^4\mathscr{F}$~are
free~$\mathscr{O}_X$-modules of rank 1,6 and 5. Also we have the
following commutative diagram:
$$\xymatrix{
\mathscr{F}^{\otimes_{\mathscr{O}_X}4}\ar[r]^{id}\ar[d]&\mathscr{F}^{\otimes_{\mathscr{O}_X}4}\ar[d]\\
S_{\mathscr{O}_X}^2 (S_{\mathscr{O}_X}^2\mathscr{F})\ar@{.>}[r]&
S_{\mathscr{O}_X}^4\mathscr{F}.}
$$
Hence there exists a canonical
homomorphism~$f:S_{\mathscr{O}_X}^2(S_{\mathscr{O}_X}^2\mathscr{F})\longrightarrow
\,S_{\mathscr{O}_X}^4\mathscr{F}$~which is defined by mapping local
sections of the
form~$m_1\otimes\,m_2\otimes\,m_3\otimes\,m_4\in\,S_{\mathscr{O}_X}^2(S_{\mathscr{O}_X}^2\mathscr{F})$~to
its image in~$S_{\mathscr{O}_X}^4\mathscr{F}$~, whose
kernel~$Ker(f)$~is free of rank 1, and we may get the short exact
sequence~$$0\longrightarrow\,Ker(f)
\longrightarrow\,S_{\mathscr{O}_X}^2(S_{\mathscr{O}_X}^2\mathscr{F})
\longrightarrow\,S_{\mathscr{O}_X}^4\mathscr{F}\longrightarrow0.$$
Using the upper remark and lemma, we
have~$$\wedge_{\mathscr{O}_X}^6(S_{\mathscr{O}_X}^2(S_{\mathscr{O}_X}^2\mathscr{F}))\cong
(\wedge_{\mathscr{O}_X}^1Ker(f))\otimes_{\mathscr{O}_X}
\wedge_{\mathscr{O}_X}^5(S_{\mathscr{O}_X}^4\mathscr{F})\cong\,
Ker(f)\otimes_{\mathscr{O}_X}\wedge_{\mathscr{O}_X}^5
(S_{\mathscr{O}_X}^4\mathscr{F})$$and~$$
(\wedge^2_{\mathscr{O}_X}\mathscr{F})^{\otimes\frac{2\cdot2\cdot6}{2}}
\cong\,Ker(f)\otimes_{\mathscr{O}_X}(\wedge_{\mathscr{O}_X}^2\mathscr{F})
^{\otimes\frac{5\cdot4}{2}}$$Hence we may
get~$Ker(f)\cong(\wedge_{\mathscr{O}_X}^2\mathscr{F})^{\otimes2}\cong\,S_{\mathscr{O}_X}^2(\wedge_{\mathscr{O}_X}^2\mathscr{F})$~and
the exact sequence as in the theorem.

Suppose~$x$~and~$y$~form a local basis
for~$\mathscr{F}|_U=\widetilde{M}$~on some affine open neighborhood
of~$U\subset\,X$, it is easy to check that~$Ker(f)|_U$~is generated
by the local sections of the
form~$(x^{\otimes2}\otimes\,y^{\otimes2}-x\otimes\,y\otimes\,x\otimes\,y)$,
the inclusion
of~$(\wedge_R^2M)^{\otimes2}\longrightarrow\,S_R^2(S_R^2M)$~is
defined by sending
~$(x\wedge\,y)\otimes(x\wedge\,y)$~to~$(x^{\otimes2}\otimes\,y^{\otimes2}-
x\otimes\,y\otimes\,x\otimes\,y)$. Next we need to check this
inclusion is independent of the choice of basis, then we may glue
together to get a canonical global homomorphism. More precisely,
if~$x'$~and~$y'$~form another basis of~$M$, then the gluing data is
reflected by a~$2\times2$~invertible matrix,
say~$\left(\begin{array}{c}x'\\y'\end{array}\right)=
\left(\begin{array}{cc}a&b\\c&d\end{array}\right)
\left(\begin{array}{c}x\\y\end{array}\right)$,where~$
\left(\begin{array}{cc}a&b\\c&d\end{array}\right)\in\,GL(2,R)$. We
need to check there exists the following equality:
$$\begin{array}{ccl}(det(f))^2(x^{\otimes2}\otimes\,y^{\otimes2}-x\otimes\,y\otimes\,x\otimes\,y)&
=&(x'^{\otimes2}\otimes\,y'^{\otimes2}-x'\otimes\,y'\otimes\,x'\otimes\,y')\\
&=&((ax+by)^{\otimes2}\otimes(cx+dy)^{\otimes2}-\\
&&(ax+by)\otimes(cx+dy)\otimes(ax+by)\otimes\\
&&\qquad(cx+dy)).
\end{array}$$
This can be proved by direct computation. Hence we may glue together
this homomorphism to get the exact sequence as described in the
theorem.

Our next aim is to construct "retraction" of the canonical
inclusion, which means we are going to construct a global
homomorphism
from~$S_{\mathscr{O}_X}^2(S_{\mathscr{O}_X}^2\mathscr{F})$~to
~$(\wedge_{\mathscr{O}_X}^2\mathscr{F})^{\otimes2}$. Note that they
are all quotient sheaf of~$\mathscr{F}^{\otimes4}$, so we may use
this original tensor sheaf~$\mathscr{F}^{\otimes4}$~to induce the
retraction. As~$\mathscr{F}^{\otimes4}$~admits permutation
group~$S_4$'s canonical action, we can define a homomorphism~$\phi$~
of~$\mathscr{F}^{\otimes4}$~as follows
$$\begin{array}{cccc}
\phi:&\mathscr{F}^{\otimes4}&\longrightarrow&\mathscr{F}^{\otimes4}\\
&f_1\otimes\,f_2\otimes\,f_3\otimes\,f_4&\mapsto&
\begin{array}{l}\sigma_{(23)}(f_1\otimes\,f_2\otimes\,f_3\otimes\,f_4)\\
+\sigma_{(234)}(f_1\otimes\,f_2\otimes\,f_3\otimes\,f_4)\\
\end{array}\end{array}$$
Where~$f_i$~are local sections
of~$\mathscr{F}$~and~$\sigma\in\,S_4$~permutate the coordinates
of~$\mathscr{F}^{\otimes4}$. Next we need to check the following
commutative diagram:
$$\xymatrix{
\mathscr{F}^{\otimes4}\ar[r]^\phi\ar[d]&\mathscr{F}^{\otimes4}\ar[d]\\
S_{\mathscr{O}_X}^2(S_{\mathscr{O}_X}^2 \mathscr{F})\ar@{.>}[r]&
S_{\mathscr{O}_X}^2(\wedge_ {\mathscr{O}_X}^2\mathscr{F})}$$By a
direct computation on local generators we know this is a commutative
diagram. Hence we get well defined homomorphism
from~$S_{\mathscr{O}_X}^2(S_{\mathscr{O}_X}^2
\mathscr{F})$~to~$(\wedge_
{\mathscr{O}_X}^2\mathscr{F})^{\otimes2}$, which we denote it
by~$\tau$.

As these homomorphisms are globally well defined, so in order to
check the composition property, we could fix a given basis,
as~$x$~and~$y$~we give before. In this case, we could give concrete
forms of~$\tau$. By a direct computation on local generators of
these sheaves, we have~$\tau\circ\,i=3$~as
$$\begin{array}{cl}
\tau(x^{\otimes2}\otimes\,y^{\otimes2}-x\otimes\,y\otimes\,x\otimes\,y)&=
x\wedge\,y\otimes\,x\wedge\,y+x\wedge\,y\otimes\,x\wedge\,y\\
&-x\wedge\,x\otimes\,y\wedge\,y-x\wedge\,y\otimes\,y\wedge\,x\\
&=3x\wedge\,y\otimes\,x\wedge\,y.\end{array}$$So if 3 is an
invertible element in~$\mathscr{O}_X$, which means if~$X$~is
an~$\mbox{Spec}(\mathbb{Z}[\frac{1}{3}])$~scheme, then the
inclusion~$i$~has retraction by making~$\tau'=\frac{1}{3}\tau$~and
the sequence is a natural split exact sequence. This finishes the
proof.
\end{proof}

\textbf{Remark:} Compare this theorem with theorem2.1, one can check
these 2 ways, globally define by permutation group's action and
locally define by local generators then glue together, are the same.

Before state the following example, we introduce some notations.
Note that
$\mathscr{F}^{\otimes4}$~admits~$S_4$'s permutation action
by permutate~$\mathscr{F}^{\otimes4}$'s coordinates, and this action
is canonical. Set~$\sigma=(12)\in\,S_4$, we will denote the
canonical automorphism of~$\mathscr{F}^{\otimes4}$~by permutating
the 1st and 2nd coordinate by~$\sigma_{(12)}$. The identity
element's action on~$\mathscr{F}^{\otimes4}$~will be denoted
by~$\sigma_{id}$.

Define the following homomorphism of~$\mathscr{F}^{\otimes4}$~as
follows:
$$\xymatrix{\mathscr{F}^{\otimes4}\ar[r]^{f_1}&
\mathscr{F}^{\otimes4}\ar[r]^{f_2}&\mathscr{F}^{\otimes4}
\ar[r]^{f_3}&\mathscr{F}^{\otimes4}}
$$
where~$f_1=\sigma_{id}-\sigma_{(23)}+\sigma_{(234)}$,~$f_2=\sigma_{(23)}-\sigma_{(234)}$~and~$f_3$~is
identity map. By a direct computation, we may get the following
homomorphism of tensor sheaves:
$$\xymatrix@C=0.5cm{\wedge_{\mathscr{O}_X}^4\mathscr{F} \ar[rr]^{\alpha_1} &&
S_{\mathscr{O}_X}^2(\wedge_{\mathscr{O}_X}^2\mathscr{F})
\ar[rr]^{\alpha_2} &&
S_{\mathscr{O}_X}^2(S_{\mathscr{O}_X}^2\mathscr{F})
\ar[rr]^{\alpha_3} &&S_{\mathscr{O}_X}^4\mathscr{F}.}$$We denote the
induced morphism by~$\alpha_1$,~$\alpha_2$~and~$\alpha_3$.

We can also define homomorphisms of~$\mathscr{F}^{\otimes4}$~as
follows:
$$\xymatrix{\mathscr{F}^{\otimes4}\ar[r]^{g_3}&
\mathscr{F}^{\otimes4}\ar[r]^{g_2}&\mathscr{F}^{\otimes4}
\ar[r]^{g_1}&\mathscr{F}^{\otimes4}}
$$
where~$g_3=\sigma_{id}+\sigma_{(23)}+\sigma_{(234)}$,~$g_2=\sigma_{(23)}+\sigma_{(234)}$~and~$f_3$~is
identity map. By a direct computation, we may get the following
homomorphism of tensor
sheaves:$$\xymatrix@C=0.5cm{S_{\mathscr{O}_X}^4\mathscr{F}
\ar[rr]^{\beta_3}&&S_{\mathscr{O}_X}^2(S_{\mathscr{O}_X}^2\mathscr{F})\ar[rr]^{\beta_2}
&& S_{\mathscr{O}_X}^2(\wedge_{\mathscr{O}_X}^2\mathscr{F})
\ar[rr]^{\beta_1}&&\wedge_{\mathscr{O}_X}^4\mathscr{F}.}$$ We denote
the induced homomorphism of tensor sheaves
by~$\beta_3$,~$\beta_2$~and~$\beta_1$.

Next we are going to prove the following theorem.
\begin{theorem}
There exists a canonical~$\mathscr{O}_X$-module complex:%
\begin{equation}\label{f6}\xymatrix@C=0.5cm{0\ar[r]&\wedge_{\mathscr{O}_X}^4\mathscr{F} \ar[rr]^{\alpha_1} &&
S_{\mathscr{O}_X}^2(\wedge_{\mathscr{O}_X}^2\mathscr{F})
\ar[rr]^{\alpha_2} &&
S_{\mathscr{O}_X}^2(S_{\mathscr{O}_X}^2\mathscr{F})
\ar[rr]^{\alpha_3}
&&S_{\mathscr{O}_X}^4\mathscr{F}\ar[r]&0.}\end{equation}

\qquad\qquad There also exists a canonical~$\mathscr{O}_X$-module complex:%
\begin{equation}\label{f7}\xymatrix@C=0.5cm{ 0\ar[r]&S_{\mathscr{O}_X}^4\mathscr{F}
\ar[rr]^{\beta_3}&&S_{\mathscr{O}_X}^2(S_{\mathscr{O}_X}^2\mathscr{F})\ar[rr]^{\beta_2}
&& S_{\mathscr{O}_X}^2(\wedge_{\mathscr{O}_X}^2\mathscr{F})
\ar[rr]^{\beta_1}&&\wedge_{\mathscr{O}_X}^4\mathscr{F}\ar[r]&0}\end{equation}
We have~$\beta_1\circ\alpha_1$~and~$\alpha_3\circ\beta_3$~are equal
to multiplication by 3. Furthermore if~$X$~is
an~$\mbox{Spec}(\mathbb{Z}[\frac{1}{3}])$~scheme, which means
if~$3$~is an invertible element in~$\mathscr{O}_X$~, the upper 2
complexes are both split exact sequences. Hence there exists a
canonical
isomorphism~$\wedge_{\mathscr{O}_X}^4\mathscr{F}\oplus\,S_{\mathscr{O}_X}^2(S_{\mathscr{O}_X}^2\mathscr{F})
\cong\,S_{\mathscr{O}_X}^2(\wedge_{\mathscr{O}_X}^2\mathscr{F})\oplus\,S_{\mathscr{O}_X}^4(\mathscr{F})$~\footnote{If~$\mathscr{F}$~is
free modules,then it can be computed by hand that the upper modules
are of the same rank
although it looks complicated.Namely we need to prove%
~$$\left(\begin{array}{c}n
\\4\\\end{array}\right)+\left(\begin{array}{c}
\left(\begin{array}{c}n+1 \\2 \\\end{array}\right)+1\\2
\\\end{array}\right)=\left(\begin{array}{c}\left(\begin{array}{c}n \\2 \\
\end{array}\right)+1 \\2 \\\end{array}\right)+\left(\begin{array}{c}
n+3 \\4 \\\end{array}\right).$$ This can be proved by easy
computation.}.
\end{theorem}

\begin{proof}
First we will check the complex properties of (\ref{f6}) and
(\ref{f7}). Note that all~$\alpha_i$~and~$\beta_i$~are induced from
homomorphisms of~$\mathscr{F}^{\otimes4}$, so in order to check the
complex property of (\ref{f6}) and (\ref{f7}). We don't have to
check on~$\alpha_i$~or~$\beta_i$, we just need to compute the
composition of~$f$~or~$g$, and prove the image in the corresponding
sheaf is trivial. Take~$\alpha_2\circ\alpha_1$~as example. We
have$$\begin{array}{rl}
f_2\circ\,f_1&=f_2(\sigma_{id}-\sigma_{(23)}+\sigma_{(234)})\\
&=(\sigma_{(23)}-\sigma_{(234)})-(\sigma_{id}-\sigma_{(24)})+(\sigma_{(34)}-\sigma_{(243)})\\
\end{array}$$
Consider its image in the
sheaf~$S_{\mathscr{O}_X}^2(S_{\mathscr{O}_X}^2\mathscr{F})$, it is
easy to check it is 0. We may also compute~$\beta_2\circ\beta_2$.
As$$\begin{array}{rl}g_2\circ\,g_3&=(\sigma_{(23)}+\sigma_{(234)})
\circ(\sigma_{id}+\sigma_{(23)}+\sigma_{(234)})\\
&=(\sigma_{(23)}+\sigma_{(234)})+(\sigma_{id}+\sigma_{(24)})+
(\sigma_{(34)}+\sigma_{(243)})\\
\end{array}$$Consider its image in the sheaf~$S_{\mathscr{O}_X}^2
(\wedge_{\mathscr{O}_X}^2\mathscr{F})$, we also
have~$\beta_2\circ\beta_3=0$. As ~$f_3$~and~$g_1$~are induced from
identity, the left can be proved by easy computation.

Next we need to check the composition properties of these
homomorphisms. Take~$\beta_1\circ\alpha_1$~as an example. We have
$$\begin{array}{cl}
g_1\circ\,f_1&=\sigma_{id}\circ(\sigma_{id}-\sigma_{(23)}+\sigma_{(234)})\\
&=\sigma_{id}-\sigma_{(23)}+\sigma_{(234)}.
\end{array}$$ Note that in~$\wedge_{\mathscr{O}_X}^4\mathscr{F}$, all these terms is equal
to~$\sigma_{id}$'s action, hence the composition is multiply by~$3$.
Similar result can be proved for~$\alpha_3\circ\beta_3$,
as
$$\begin{array}{cl}
f_3\circ\,g_3&=\sigma_{id}\circ(\sigma_{id}
+\sigma_{(23)}+\sigma_{(234)})\\
&=\sigma_{id} +\sigma_{(23)}+\sigma_{(234)}
\end{array}$$
Note that the image is in the
sheaf~$S_{\mathscr{O}_X}^4\mathscr{F}$, hence they are all equal
to~$\sigma_{id}$'s action, which means the composition is multiply
by~$3$.

When~$X$~is an~$\mbox{Spec}(\mathbb{Z}[\frac{1}{3}])$~scheme, which
means if~$3$~is an invertible element in~$\mathscr{O}_X$, we
have~$\alpha_1$~and~$\beta_3$~are
injective,~$\alpha_3$~and~$\beta_1$~are surjective. This proves the
exactness of~$(\ref{f6})$~and~$(\ref{f7})$~at the point
of~$\wedge_{\mathscr{O}_X}^4\mathscr{F}$~and~$S_{\mathscr{O}_X}^4\mathscr{F}$~.
As~$$\begin{array}{cl}
g_2\circ\,f_2&=(\sigma_{(23)}+\sigma_{(234)})\circ(\sigma_{(23)}-\sigma_{(234)})\\
&=(\sigma_{id}-\sigma_{(34)})+(\sigma_{(24)}-\sigma_{(243)})
\end{array}$$
If we consider its image in the
sheaf~$S_{\mathscr{O}_X}^2(\wedge_{\mathscr{O}_X}^2\mathscr{F})$, we
have2$$\begin{array}{cl}&(\sigma_{id}-\sigma_{(34)})+(\sigma_{(24)}-\sigma_{(243)})\\
=&3\sigma_{id}-(\sigma_{id}-\sigma_{(24)}+\sigma_{(243)})
\end{array}$$Which is just~$3\sigma_{id}-\alpha_1$, hence we
get~$\beta_2\circ\alpha_2=3id-\alpha_1$. Hence for any local
section~$f$~of the subsheaf~$Ker(\alpha_2)$~, we have local
section~$f'\in\wedge_{\mathscr{O}_X}^4\mathscr{F}$~such
that~$3f-\alpha_1(f')=0$~, namely we
have~$f=\frac{1}{3}\alpha_1(f')$~. This
means~$\mbox{Ker}(\alpha_2)=\mbox{im}(\alpha_1)$~. So the exactness
at~$S_{\mathscr{O}_X}^2(\wedge_{\mathscr{O}_X}^2\mathscr{F})$~in the
complex~$(\ref{f6})$~is exact. Similar result can be proved for
exactness
at~$S_{\mathscr{O}_X}^2(S_{\mathscr{O}_X}^2\mathscr{F})$~for the
complex~$(\ref{f7})$. At last we need to check the exactness
property at~$S_{\mathscr{O}_X}^2(S_{\mathscr{O}_X}^2\mathscr{F})$~in
the complex~$(\ref{f6})$~. As
$$g_3\circ\,f_3=(\sigma_{id}+\sigma_{(23)}+\sigma_{(234)})\circ\sigma_{id}=\sigma_{id}+\sigma_{(23)}+\sigma_{(234)}.$$Consider
its image in the
sheaf~$S_{\mathscr{O}_X}^2(S_{\mathscr{O}_X}^2\mathscr{F})$, we have
the image is equal to~$\sigma_{id}+\sigma_{(23)}+\sigma_{(24)}$,
also note that
$$\begin{array}{c}
\alpha_2=\sigma_{(23)}-\sigma_{(234)}\\
\alpha_2\circ\sigma_{(24)}=\sigma_{(243)}-\sigma_{(34)}\\
\alpha_2\circ\sigma_{(243)}=\sigma_{(24)}-\sigma_{id}
\end{array}
$$
So in the
sheaf~$S_{\mathscr{O}_X}^2(S_{\mathscr{O}_X}^2\mathscr{F})$, there
images are equal
to~$\sigma_{(23)}-\sigma_{(24)}$,~$\sigma_{(23)}-\sigma_{id}$~and
~$\sigma_{(24)}-\sigma_{id}$. Hence we
have~$\beta_3\circ\alpha_3=3\sigma_{id}+\alpha_2(\sigma_{(24)}+\sigma_{(243)})$.
As~$\sigma_{(24)}$~and~$\sigma_{(243)}$~can be regarded as
automorphisms, hence we may
write~$\beta_3\circ\alpha_3=3\sigma_{id}+\alpha_2$, so the exactness
of (\ref{f6})
at~$S_{\mathscr{O}_X}^2(S_{\mathscr{O}_X}^2\mathscr{F})$~is proved.
Similar result can be proved
at~$S_{\mathscr{O}_X}^2(\wedge_{\mathscr{O}_X}^2\mathscr{F})$
in the
complex~$(\ref{f7})$~. This ends the proof.
\end{proof}

\end{document}